\documentclass[11pt]{amsart}
\usepackage{times}
\usepackage[pdftex]{graphicx}
\usepackage{amssymb,amsfonts,amsmath,amsthm}
\usepackage{float}

\usepackage{color}

\newtheorem{theorem}{Theorem}
\newtheorem{proposition}{Proposition}
\newtheorem{lemma}{Lemma}

\newtheorem{remark}{Remark}

\def\reff#1{{\rm(\ref{#1})}}

\newcommand\cP{\mathcal P}

\newcommand\bA {\mathbb A}

\newcommand\EE {\mathbb E}
\newcommand\FF {\mathbb F}
\newcommand\HH {\mathbb H}

\newcommand\RR {\mathbb R}

\newcommand\PP {\mathbb P}

\newcommand\cA {\mathcal A}

\newcommand\cE {\mathcal E}
\newcommand\cF {\mathcal F}

\def\u{\underline}

\newcommand\PXt {{\mathbb P}_{X_t}}

\begin{document}

\title[Mean Field FBSDEs]{Mean Field Forward-Backward Stochastic Differential Equations}

\author{Ren\'e Carmona}
\address{ORFE, Bendheim Center for Finance, Princeton University,
Princeton, NJ  08544, USA.}
\email{rcarmona@princeton.edu}
\thanks{Partially supported  by NSF: DMS-0806591}

\author{Francois Delarue}
\address{Laboratoire Jean-Alexandre Dieudonn\'e 
Universit� de Nice Sophia-Antipolis 
Parc Valrose 
06108 Cedex 02, Nice, FRANCE}
\email{Francois.Delarue@unice.fr}

\subjclass[2000]{Primary }


\date{November 12, 2012}

\begin{abstract}
The purpose of this note is to provide an existence result for the solution of fully coupled Forward Backward Stochastic Differential Equations (FBSDEs) 
of the mean field type. These equations occur in the study of mean field games and the optimal control of dynamics of the
McKean Vlasov type.
\end{abstract}

\maketitle

\section{\textbf{Introduction}}

Stochastic differential equations of the McKean - Vlasov type are It\^o's stochastic differential equations where the coefficients depend upon the marginal distribution of the solution. In their partial differential form, they were introduced by Mark Kac's in his analysis  of the Boltzmann equation for the density of particules in kinetic theory of dilute monatomic gases, and a  toy model for the Vlasov kinetic equation of plasma (see  \cite{Kac,Kac.book}).

The purpose of this note is to provide an existence result for the solution of Forward Backward Stochastic Differential Equations (FBSDEs) 
of the McKean-Vlasov type. Following the wave of interest created by the pathbreaking work of Lasry and Lions on mean field games \cite{MFG1,MFG2,MFG3}, simple forms of Backward Stochastic Differential Equations (BSDEs) of McKean Vlasov type have been introduced and called of  \emph{mean field type}. 
Fully coupled  FBSDEs are typically more involved and more difficult to solve than BSDEs. FBSDEs of mean field type occur naturally in the probabilistic analysis of mean field games and the optimal control of dynamics of the McKean Vlasov type as considered in \cite{CarmonaDelarue2,CarmonaDelarue3}. See also 
\cite{Bensoussanetal, CarmonaDelarueLachapelle,Yong3} for the particular case of Linear Quadratic (LQ) models. Detailed explanations on how these FBSDEs occur in these contexts and the particular models which were solved are given in Section \ref{se:applications} below.

The existence proofs given in  \cite{CarmonaDelarue2} and \cite{CarmonaDelarue3} depend heavily on the fact that the problems at hand are in fact stochastic control problems and FBSDEs are derived from an application of a version of the stochastic maximum principle, and the compactness estimates are derived from the linear nature of the forward dynamics and strong convexity properties of the cost functions of the stochastic optimization problems. The purpose of this note is to provide a general existence result which does not depend upon strong linearity and convexity assumptions. Such an existence result is proven in Section \ref{se:fbsde}. The proof relies on Schauder's fixed point theorem used in  appropriate  spaces of functions an measures. A short Section \ref{se:applications} concludes with a short discussion of applications to mean field games and control of McKean-Vlasov dynamics studied in  \cite{CarmonaDelarue2} and \cite{CarmonaDelarue3}.

\section{\textbf{Solvability of Forward-Backward Systems of McKean-Vlasov Type}}
\label{se:fbsde}

\subsection{First Notation}

All the processes considered in this note are defined on a probability space $(\Omega,\cF,\PP)$ on which an $m$-dimensional Wiener process $\u W=(W_t)_{0\le t\le T}$ is defined. For each random variable/vector or stochastic process $X$, we denote
by $\PP_X$ the law (alternatively called the distribution) of $X$.  We shall denote by $\FF=(\cF_t)_{0\le t\le T}$ the filtration of $\u W$ and by $\HH^{2,n}$ the Hilbert space
$$
\HH^{2,n}:=\Big\{Z\in\HH^{0,n};\;\EE\int_0^T|Z_s|^2ds<\infty\Big\}
$$
where $\HH^{0,n}$ stands for the collection of all $\RR^n$-valued progressively measurable processes on $[0,T]$. For any measurable space $(E,\cE)$, we denote by $\cP(E)$ the space of probability measures on $(E,\cE)$ assuming that the $\sigma$-field $\cE$  is understood. When $E$ is a normed space (most often $E=\RR^d$ in what follows), we denote by $\cP_p(E)$ the subspace of $\cP(E)$ of the probability measures of order $p$, namely those elements of $\cP(E)$ which integrate the $p$-th power of the distance to a fixed point (whose choice is irrelevant in the definition of $\cP_p(E)$).
For each $p\ge 1$, if $\mu$ and $\mu'$ are probability measures of order $p$, $W_p(\mu,\mu')$ denotes the $p$-Wasserstein's distance defined as
$$
W_p(\mu,\mu')=\inf\left\{\left[\int |x-y|_{E}^p\pi(dx,dy)\right]^{1/p};\;\pi\in\cP_p(E\times E) \text{ with marginals } \mu \text{ and } \mu'\right\}.
$$
Notice that if $X$ and $X'$ are random variables of order $2$ with values in $E$, then by definition we have 
$$
W_2(\PP_X,\PP_{X'})\le [\EE|X-X'|_{E}^2]^{1/2}.
$$

\subsection{Assumptions and Statement of the Main Existence Result}

Our goal is to solve fully coupled McKean-Vlasov  forward-backward systems of the general form:
\begin{equation}
\label{fo:fbsde}
\begin{split}
&dX_{t} = B\bigl(t,X_{t},Y_{t},Z_{t},\PP_{(X_{t},Y_{t})}\bigr) dt + \Sigma\bigl(t,X_{t},Y_{t},\PP_{(X_{t},Y_{t})}\bigr) dW_{t}
\\
&dY_{t} = -F\bigl(t,X_{t},Y_{t},Z_{t},\PP_{(X_{t},Y_{t})}\bigr) dt + Z_{t} dW_{t}, \quad 0 \leq t \leq T,
\end{split}
\end{equation}
with initial condition $X_{0}=x_{0}$ for a given deterministic point $x_{0} \in \RR^d$, and  terminal condition $Y_{T} = G(X_{T},\PP_{X_{T}})$. Here, the unknown processes $(\u X,\u Y,\u Z)$ are of dimensions $d$, $p$ and $p \times m$ respectively, the coefficients $B$ and $F$ map 
$[0,T] \times \RR^d \times \RR^p \times \RR^{p \times m} \times {\mathcal P}_2(\RR^{d} \times \RR^{p})$ 
into $\RR^d$ and $\RR^p$ respectively, while the coefficient $\Sigma$ maps $[0,T] \times \RR^d \times \RR^p \times {\mathcal P}_2(\RR^{d} \times \RR^{p})$  into $\RR^{d \times m}$ and the function $G$ giving the terminal condition maps $\RR^d \times {\mathcal P}_2(\RR^d)$ into $\RR^p$, all these functions being assumed to be Borel-measurable. Recall that  
the spaces ${\mathcal P}_2(\RR^d \times \RR^p)$ and ${\mathcal P}_{2}(\RR^{d})$ are assumed to be endowed with the 
topology of the $2$-Wasserstein distance $W_2$ defined earlier. 

\vskip 4pt
We now state the standing assumptions of the paper. We assume that, for any $t \in [0,T]$, the coefficients $B(t,\cdot)$, $F(t,\cdot)$, $\Sigma(t,\cdot)$ and $G$ satisfy:
\vskip 4pt
(A1). The coefficients $B(t,\cdot)$, $F(t,\cdot)$, $\Sigma(t,\cdot)$ and $G$ are
Lipschitz continuous, the Lipschitz property with respect to the measure arguments holding for the $2$-Wasserstein distance. Precisely, we assume that there exists a constant $L \geq 1$, such that for any $t \in [0,T]$,  $x,x' \in \RR^d$, $y,y' \in \RR^p$, $z,z' \in \RR^{p \times m}$ and $\mu,\mu' \in {\mathcal P}_{2}(\RR^d \times \RR^p)$, 
\begin{eqnarray*}
&&\vert B(t,x',y',z',\mu') - B(t,x,y,z,\mu) \vert +\vert F(t,x',y',z',\mu') - F(t,x,y,z,\mu) \vert\\
&&\phantom{?????}+\vert \Sigma(t,x',y',\mu') - \Sigma(t,x,y,\mu) \vert+\vert G(x',\mu') - G(x,\mu) \vert\\
&&\phantom{??????????????????????}\leq L \bigl[ \vert (x,y,z) - (x',y',z') \vert + W_{2}(\mu,\mu') \bigr].
\end{eqnarray*}

\vskip 4pt
(A2) The functions $\Sigma$ and $G$ are bounded, the common bound being also denoted by $L$. Moreover,  for any $t \in [0,T]$, $x \in \RR^d$, $y \in \RR^p$, $z \in \RR^{p \times m}$ and $\mu \in {\mathcal P}_{2}(\RR^{d \times p})$, 
\begin{equation*}
\begin{split}
&\vert B(t,x,y,z,\mu) \vert \leq L \biggl[ 1 + \vert x \vert + \vert y \vert + \vert z \vert + \biggl( \int_{\RR^d \times \RR^p}  \vert (x',y') \vert^2 d\mu(x',y')
\biggr)^{1/2} \biggr],
\\
&\vert F(t,x,y,z,\mu) \vert \leq L \biggl[ 1 + \vert y \vert + \biggl( \int_{\RR^d \times \RR^p} \vert y' \vert^2 d\mu(x',y')
\biggr)^{1/2} \biggr].
\end{split}
\end{equation*}

\vskip 4pt
(A3) The function $\Sigma$ is uniformly elliptic in the sense that 
for any $t \in [0,T]$, $x \in \RR^d$, $y \in \RR^p$ and 
$\mu \in {\mathcal P}_{2}(\RR^d \times \RR^p)$ the following inequality holds 
\begin{equation*}
\Sigma(t,x,y,\mu)\Sigma(t,x,y,\mu)^\dagger  \geq L^{-1} I_d 
\end{equation*}
in the sense of symmetric matrices, where $I_d$ is the $d$-dimensional identity matrix. Here and throughout the paper, we use the exponent ${}^\dagger$ to denote the transpose of a matrix. Moreover, the function $[0,T] \ni t \hookrightarrow \Sigma(t,0,0,\delta_{(0,0)})$ is also assumed to be continuous.

\vskip 4pt
We can now state the main result of the paper.

\begin{theorem}
\label{th:mkv_fbsde}
Under (A1--3), the FBSDE \eqref{fo:fbsde} has a solution.
\end{theorem}

The strategy of the proof was sketched in a simpler setting in \cite{CarmonaDelarueLachapelle}. We review this strategy before giving the details. Because of the Markovian nature of the set-up, we expect $Y_{t}$ and $X_{t}$ to be connected by a deterministic relationship of the form 
$Y_{t} = \varphi(t,X_{t})$, $\varphi$ being a function from $[0,T] \times \RR^d$ into $\RR^p$ usually called the FBSDE value function. If this is the case, the law of the pair $(X_{t},Y_{t})$ is entirely determined by the law of $X_{t}$ since the distribution $\PP_{(X_t,Y_t)}$ of $(X_{t},Y_{t})$ is equal to $(I_{d},\varphi(t,\cdot))(\PP_{X_{t}})$. For a random variable $X$ with values in $\RR^d$ and for a measurable mapping $\psi$ from $\RR^d$ into $\RR^p$, we shall
denote by $\psi \diamond \PP_{X}$ the image of the distribution $\PP_X$ of $X$ under the map $(I_{d},\psi): \RR^d\ni x\hookrightarrow (x,\psi(x))\in\RR^d\times\RR^p$. With this notation in hand, it is natural to look for a function $\varphi : [0,T] \times \RR^d \rightarrow \RR^p$ such that 
\begin{equation}
\label{fo:fbsde'}
\begin{split}
&dX_{t} = B\bigl(t,X_{t},Y_{t},Z_{t},\varphi(t,\cdot) \diamond \PP_{X_{t}}\bigr) dt + \Sigma\bigl(t,X_{t},Y_{t},\varphi(t,\cdot) \diamond \PP_{X_{t}}\bigr) dW_{t}
\\
&d Y_{t} = - F\bigl(t,X_{t},Y_{t},Z_{t},\varphi(t,\cdot) \diamond \PP_{X_{t}}\bigr) dt + Z_{t} dW_{t}, \quad 0 \leq t \leq T,
\end{split}
\end{equation}
under the constraint that $Y_t=\varphi(t,X_{t})$ for $t \in [0,T]$. 
Translating the above into a nonlinear PDE, $\varphi$ appears as the solution to a nonlinear PDE of the McKean-Vlasov type. The strategy we use below  consists in recasting the stochastic system \eqref{fo:fbsde'} into a well-posed fixed point problem over the arguments $(\varphi,(\PP_{X_{t}})_{0 \leq t \leq T})$. The first step is to use $\varphi(t,\cdot) \diamond \PP_{X_{t}}$ as an input and then to solve \eqref{fo:fbsde'} as a standard FBSDE. In order to do so, we  use known existence results for standard FBSDEs. 

\begin{remark}
The proof of Theorem \ref{th:mkv_fbsde} given in this note can be used to derive existence of a solution of \eqref{fo:fbsde} when the law $\PP_{(X_{t},Y_{t})}$ is replaced by $\PP_{(X_{t},Y_{t},Z_t)}$. Indeed, $Z_t$ is also given by a function of $X_t$ in the same way $Y_t$ is, since 
$Z_t= v(t,X_{t})$ with $v(t,x)
= \partial_x u(t,x) \Sigma(t,x,u(t,x),u(t,\cdot) \diamond \PP_{X_{t}})$ whenever $Y_t=u(t,X_t)$ (we prove below that $\partial_{x}u$ makes sense). We refrain from giving the details to keep the technicalities to a minimum, and because we do not know of a practical application of such a generalization.
\end{remark}

\subsection{Preliminary }
\label{sub:bdedcase}

 Our fixed point argument relies on the following lemma which puts together the existence and uniqueness result contained in Theorem 2.6 of Delarue 
\cite{Delarue02} and the control of the FBSDE value function provided by Corollary 2.8 of \cite{Delarue02}:

\begin{lemma}
\label{lem:1}
On the top of (A1-3), let us also assume that $B$ and $F$ are bounded by $L$. Then, given a probability measure $\nu' \in {\mathcal P}_{2}(\RR^d)$, a deterministic continuous function $\nu : [0,T] \ni t \hookrightarrow \nu_{t} \in {\mathcal P}_{2}(\RR^{d} \times \RR^{p})$, and an initial condition $(t,x) \in [0,T] \times \RR^d$,  the forward-backward system
\begin{equation}
\label{fo:regular_fbsde}
\begin{split}
&dX_{s} = B\bigl(s,X_{s},Y_{s},Z_{s},\nu_{s}\bigr) ds + \Sigma(s,X_{s},Y_{s},\nu_{s}) dW_{s}
\\
&dY_{s} = - F\bigl(s,X_{s},Y_{s},Z_{s},\nu_{s}\bigr) ds + Z_{s} dW_{s}, \quad t \leq s \leq T,
\end{split}
\end{equation}
with $X_{t}=x$ as initial condition and $Y_{T}=G(X_{T},\nu')$ as terminal condition, has a unique solution, denoted by 
$(X_{s}^{t,x},Y_{s}^{t,x},Z_{s}^{t,x})_{t \leq s \leq T}$. Moreover, 
the FBSDE value function $u : [0,T] \times \RR^d \ni (t,x) \hookrightarrow Y_{t}^{t,x}\in\RR^p$ is bounded by a
constant $\gamma$ depending only upon $T$ and $L$, and is $1/2$-H\"older continuous in time
and Lipschitz continuous in space in the sense that:
\begin{equation*}
\vert u(t,x) - u(t',x') \vert \leq \Gamma \bigl( \vert t - t' \vert^{1/2} + \vert x-x' \vert \bigr),
\end{equation*}
for some constant $\Gamma$ only depending upon $T$ and $L$. In particular, both $\gamma$ and $\Gamma$ are independent of $\nu'$ and $\nu$. Finally, it holds $Y_{s}^{t,x}=u(s,X_{s}^{t,x})$
for any $t \leq s \leq T$. 
\end{lemma}

The boundedness assumption on $B$ and $F$ is stronger than what is necessary for the result of Lemma \ref{lem:1} to hold. For instance, the result of \cite{Delarue02} only requires that the bound (A2) holds without $|x|$ in the right hand side. We will not need this extra level of generality. We use this existence result in the following way. We start with a bounded continuous function $\varphi$ from $[0,T]\times\RR^d$ into $\RR^p$ and a probability measure  $\mu$ on $C([0,T],\RR^d)$ which we want to think of as the law $\PP_X$ of the solution, we denote by $\mu_t$ its marginal distributions on $\RR^d$, and we apply the above existence result for \eqref{fo:regular_fbsde} to $\nu'=\mu_T$ and $\nu_t=\varphi(t,\cdot)\diamond\mu_t$ for $t\in[0,T]$ and solve:
\begin{equation}
\label{fo:our_fbsde}
\begin{split}
&dX_{t} = B\bigl(t,X_{t},Y_{t},Z_{t},\varphi(t,\cdot) \diamond \mu_{t}\bigr) dt + \Sigma(t,X_{t},Y_{t},\varphi(t,\cdot) \diamond \mu_{t}) dW_{t}
\\
&d Y_{t}= - F\bigl(t,X_{t},Y_{t},Z_{t},\varphi(t,\cdot) \diamond \mu_{t}\bigr) dt + Z_{t} dW_{t}, \quad 0 \leq t \leq T,
\end{split}
\end{equation}
with the terminal condition $Y_{T} = g(X_{T},\mu_{T})$ and a prescribed initial condition $X_{0}= x_{0} \in \RR^d$.  The following estimate will be instrumental in the proof of the main result.

\begin{lemma}
\label{le:W2est} 
On the top of (A1-3), let us also assume that $B$ and $F$ are bounded by $L$. Then, there exists a positive constant $\Gamma$, depending on $T$ and $L$ only, such that, for any inputs $(\varphi,\mu)$ and $(\varphi',\mu')$ as above, the processes
$(X,Y,Z)$ and $(X',Y',Z')$ obtained by solving \eqref{fo:our_fbsde} with $(\varphi,\mu)$ and $(\varphi',\mu')$ respectively, satisfy
\begin{equation}
\label{fo:W2est}
\begin{split}
&{\mathbb E} \bigl[ \sup_{0 \leq t \leq T}\vert X_{t} - X_{t}' \vert^2 \bigr]
+ {\mathbb E} \bigl[ \sup_{0 \leq t \leq T}\vert Y_{t} - Y_{t}' \vert^2 \bigr]
+ {\mathbb E} \int_0^T \vert Z_{t} - Z_{t}' \vert^2 dt
\\
&\hspace{15pt} \leq \Gamma \biggl(  W_{2}(\mu_{T},\mu_{T}')^2 + \int_{0}^T W_{2}(\varphi(t,\cdot) \diamond \mu_{t},\varphi'(t,\cdot) \diamond \mu_{t}')^2 dt \biggr).
\end{split}
\end{equation}
\end{lemma}

\begin{proof}
For small time $T>0$, this estimate follows immediately from the main estimate Theorem 1.3 p. 218 of  \cite{Delarue02} and the Lipschitz assumption (A1). We only need to show that one can extend it to arbitrarily large values of $T$. Notice that Lemma \ref{lem:1} gives the existence of the FBSDE values functions $u$ and $u'$ such that $Y_t=u(t,X_t)$ and $Y'_t=u(t,X'_t)$ for all $t\in[0,T]$.

As in Corollary 2.8 of  \cite{Delarue02}, we choose a regular subdivision $0=T_0<T_1<\cdots<T_{N-1}<T_N=T$ so that the common length of the intervals $[T_i,T_{i+1}]$ is small enough in order to apply the main estimate Theorem 1.3 p. 218 of  \cite{Delarue02}. For any $i\in\{0,\cdots,N-1\}$ we have:
\begin{equation}
\label{fo:Ei}
\begin{split}
&\EE\bigl[ \sup_{T_i \leq t \leq T_{i+1}}\vert X_{t} - X_{t}' \vert^2 \bigr]
+ {\mathbb E} \bigl[ \sup_{T_i \leq t \leq T_{i+1}}\vert Y_{t} - Y_{t}' \vert^2 \bigr]
+ {\mathbb E} \int_{T_i}^{T_{i+1}} \vert Z_{t} - Z_{t}' \vert^2 dt
\\
&\hspace{25pt} \leq \Gamma \biggl( \EE[|X_{T_i}-X'_{T_i}|^2]+\EE[|u(T_{i+1},X_{T_{i+1}})-u' (T_{i+1},X_{T_{i+1}})|^2]\\
&\hspace{45pt} + \int_{T_i}^{T_{i+1}} W_{2}(\varphi(t,\cdot) \diamond \mu_{t},\varphi'(t,\cdot) \diamond \mu_{t}')^2 dt \biggr).
\end{split}
\end{equation}
We first consider the last interval $[T_{N-1},T_N]$ corresponding to the case $i=N-1$. Since $T_N=T$ we have $u(T,\cdot)=G(\cdot,\mu_T)$ and $u'(T,\cdot)=G(\cdot,\mu'_T)$ so that using the Lipschitz property of $G$ we get:
\begin{equation*}
\begin{split}
&\EE\bigl[ \sup_{T_{N-1} \leq t \leq T}\vert X_{t} - X_{t}' \vert^2 \bigr]
+ {\mathbb E} \bigl[ \sup_{T_{N-1} \leq t \leq T}\vert Y_{t} - Y_{t}' \vert^2 \bigr]
+ {\mathbb E} \int_{T_{N-1}}^{T} \vert Z_{t} - Z_{t}' \vert^2 dt
\\
&\hspace{15pt} \leq \Gamma \biggl( \EE[|X_{T_{N-1}}-X'_{T_{N-1}}|^2]+W_2(\mu_T,\mu'_T)^2 + \int_{T_{N-1}}^{T} W_{2}(\varphi(t,\cdot) \diamond \mu_{t},\varphi'(t,\cdot) \diamond \mu_{t}')^2 dt \biggr),
\end{split}
\end{equation*}
this estimate being true for all $\varphi$, $\varphi'$, $\mu$, $\mu'$ and all possible initial conditions for the processes $\u X$ and $\u X'$.  Note that we can assume $\Gamma>1$ without any loss of generality, and allow the value of $\Gamma$ to change from line to line as long as this new value depends only upon $T$ and $L$. Since the FBSDE value function $u$ (resp. $u'$) depends only upon $\varphi$ and $\mu$ (resp. $\varphi'$ and $\mu'$), we can choose to keep $\varphi$, $\varphi'$, $\mu$, $\mu'$, and set $X_{T_{N-1}}=X'_{T_{N-1}}=x$ for an arbitrary $x\in\RR^d$. Then the above inequality implies
$$
\sup_{x\in\RR^d}|u(T_{N-1},x)-u'(T_{N-1},x)|^2\le \Gamma \biggl( W_2(\mu_T,\mu'_T)^2 + \int_{T_{N-1}}^{T} W_{2}(\varphi(t,\cdot) \diamond \mu_{t},\varphi'(t,\cdot) \diamond \mu_{t}')^2 dt \biggr).
$$
We can now plug this estimate into inequality \eqref{fo:Ei} with $i=N-2$ to get:
\begin{equation*}
\begin{split}
&\EE\bigl[ \sup_{T_{N-2} \leq t \leq T_{N-1}}\vert X_{t} - X_{t}' \vert^2 \bigr]
+ {\mathbb E} \bigl[ \sup_{T_{N-2} \leq t \leq T_{N-1}}\vert Y_{t} - Y_{t}' \vert^2 \bigr]
+ {\mathbb E} \int_{T_{N-2}}^{T_{N-1}} \vert Z_{t} - Z_{t}' \vert^2 dt
\\
&\hspace{15pt} \leq \Gamma \biggl( \EE[|X_{T_{N-2}}-X'_{T_{N-2}}|^2]+W_2(\mu_T,\mu'_T)^2 + \int_{T_{N-2}}^{T} W_{2}(\varphi(t,\cdot) \diamond \mu_{t},\varphi'(t,\cdot) \diamond \mu_{t}')^2 dt \biggr).
\end{split}
\end{equation*}
As before, we can write what this estimate gives if we keep $\varphi$, $\varphi'$, $\mu$, $\mu'$, and set $X_{T_{N-2}}=X'_{T_{N-2}}=x$ for an arbitrary $x\in\RR^d$. 
$$
\sup_{x\in\RR^d}|u(T_{N-2},x)-u'(T_{N-2},x)|^2\le \Gamma \biggl( W_2(\mu_T,\mu'_T)^2 + \int_{T_{N-2}}^{T} W_{2}(\varphi(t,\cdot) \diamond \mu_{t},\varphi'(t,\cdot) \diamond \mu_{t}')^2 dt \biggr).
$$
Plugging this estimate into inequality \eqref{fo:Ei} with $i=N-3$ we get:
\begin{equation*}
\begin{split}
&\EE\bigl[ \sup_{T_{N-3} \leq t \leq T_{N-2}}\vert X_{t} - X_{t}' \vert^2 \bigr]
+ {\mathbb E} \bigl[ \sup_{T_{N-3} \leq t \leq T_{N-2}}\vert Y_{t} - Y_{t}' \vert^2 \bigr]
+ {\mathbb E} \int_{T_{N-3}}^{T_{N-2}} \vert Z_{t} - Z_{t}' \vert^2 dt
\\
&\hspace{15pt} \leq \Gamma \biggl( \EE[|X_{T_{N-3}}-X'_{T_{N-3}}|^2]+W_2(\mu_T,\mu'_T)^2 + \int_{T_{N-3}}^{T} W_{2}(\varphi(t,\cdot) \diamond \mu_{t},\varphi'(t,\cdot) \diamond \mu_{t}')^2 dt \biggr).
\end{split}
\end{equation*}
Iterating and summing up these estimates we get (as before the value of the constants can change from line to line)
\begin{equation*}
\begin{split}
&{\mathbb E} \bigl[ \sup_{0 \leq t \leq T}\vert X_{t} - X_{t}' \vert^2 \bigr]
+ {\mathbb E} \bigl[ \sup_{0 \leq t \leq T}\vert Y_{t} - Y_{t}' \vert^2 \bigr]
+ {\mathbb E} \int_0^T \vert Z_{t} - Z_{t}' \vert^2 dt
\\
&\hspace{15pt} \leq \Gamma \sum_{i=0}^{N-1}\biggl(\EE[|X_{T_i}-X'_{T_i}|^2]+  W_{2}(\mu_{T},\mu_{T}')^2 + \int_{T_i}^T W_{2}(\varphi(t,\cdot) \diamond \mu_{t},\varphi'(t,\cdot) \diamond \mu_{t}')^2 dt \biggr),
\end{split}
\end{equation*}
from which we get the desired estimate \eqref{fo:W2est} after noticing that for each $i\ge 1$, we have:
\begin{equation*}
\begin{split}
&\EE[|X_{T_i}-X'_{T_i}|^2]\le\EE[\sup_{T_{i-1}\le t\le T_i}|X_t-X'_t|^2]\\
&\leq \Gamma\bigg(\EE[|X_{T_{i-1}}-X'_{T_{i-1}}|^2]+W_{2}(\mu_{T},\mu_{T}')^2 + \int_{T_{i-1}}^T W_{2}(\varphi(t,\cdot) \diamond \mu_{t},\varphi'(t,\cdot) \diamond \mu_{t}')^2 dt \bigg).
\end{split}
\end{equation*}
from which we easily conclude.
\end{proof}

\vskip 4pt
We shall estimate the integral in the right hand side of \eqref{fo:W2est} from the remark:
\begin{equation}
\label{fo:preW2est}
\begin{split}
W_{2}(\varphi(t,\cdot) \diamond \mu_{t},\varphi'(t,\cdot) \diamond \mu_{t}')
&\leq C \bigg[ W_{2}(\mu_{t},\mu_{t}') + W_{2}(\varphi(t,\cdot)(\mu_{t}),\varphi(t,\cdot)(\mu_{t}')) 
\\
&\hspace{55pt}
+ \biggl( \int_{\RR^d} \vert (\varphi - \varphi')(t,x) \vert^2 d \mu_{t}'(x) \biggr)^{1/2}\bigg].
\end{split}
\end{equation}

\subsection{Fixed Point Argument in the Bounded Case}
\label{sub:bdedcase}
In this subsection we still assume that the coefficients $B$ and $F$ are bounded by the constant $L$.
For any  bounded continuous function $\varphi: [0,T] \times \RR^d \rightarrow \RR^p$ and for any 
probability measure $\mu \in {\mathcal P}_{2}({\mathcal C}([0,T];\RR^d))$, if we denote by $\mu_t$ the time $t$ marginal of $\mu$, the map $[0,T]\ni t\hookrightarrow \varphi(t,\cdot)\diamond\mu_t\in {\mathcal P}_{2}(\RR^d\times\RR^p)$ is continuous. So by Lemma \ref{lem:1}, there exists a unique triplet $(X_{t},Y_{t},Z_{t})_{0 \leq t \leq T}$ satisfying \reff{fo:our_fbsde}. 
Moreover, there exists a bounded and continuous mapping $u$ from  $[0,T] \times \RR^d$ into $\RR^p$ such that 
$Y_{t} = u(t,X_{t})$. This maps the input $(\varphi,\mu)$ into the output $(u,\PP_{X})$ and our goal is to find a fixed point for this map. 
We shall take advantage of the a-priori $L^{\infty}$ bound on $u$ to restrict the choice of the functions $\varphi$ to the set:
\begin{equation}
\label{fo:E1}
E_{1} = \bigl\{ \varphi \in {\mathcal C}([0,T] \times \RR^d;\RR^p) ; \quad \forall (t,x) \in [0,T] \times \RR^d, \quad \vert \varphi(t,x) \vert \leq \gamma \bigr\}.
\end{equation}
Similarly, since the drift $B$ and the volatility $\Sigma$ are uniformly bounded, the fourth moment of the supremum $\sup_{0\le t\le T}|X_t|$ is bounded by a constant depending only upon the bounds of $B$ and $\Sigma$. Consequently, we shall choose the input measure
$\mu$ in the set:
\begin{equation}
\label{fo:E2}
E_{2} = \biggl\{ \mu \in {\mathcal P}_4\bigl({\mathcal C}([0,T];\RR^d)\bigr) ;
\int_{{\mathcal C}([0,T];\RR^d)} \;\sup_{0\leq t \leq T} \vert w_{t} \vert^4 \;d\mu(w) \leq \gamma'
 \biggr\}
\end{equation}
for $\gamma'$ appropriately chosen.
We then denote by $E$ the Cartesian product $E=E_{1} \times E_{2}$. We view $E$ as a subset of the product vector space $V = V_{1} \times V_{2}$, where $V_{1} = {\mathcal C}_{b}([0,T] \times \RR^d;\RR^p)$ stands for the space of bounded continuous functions from $[0,T] \times \RR^d$ into $\RR^p$, and  $V_{2} = {\mathcal M}_{b}({\mathcal C}([0,T];\RR^d))$ for the space of finite signed measures on the space ${\mathcal C}([0,T];\RR^d)$ endowed with the Borel $\sigma$-field generated by the topology of uniform convergence. On $V_{1}$, we use the exponentially weighted supremum-norm 
\begin{equation*}
\| h \|_{1} = \sup_{(t,x) \in [0,T] \times \RR^d} e^{- \vert x \vert} | h(t,x)|,
\end{equation*}
and on $V_{2}$,  the Kantorovitch-Rubinstein norm 
\begin{equation*}
\| \mu \|_{2} = \sup \biggl\{ \int_{{\mathcal C}([0,T];\RR^d)} F d\mu ; \quad F \in {\rm Lip}_{1}({\mathcal C}([0,T];\RR^d)), \ \sup_{h \in {\mathcal C}([0,T];\RR^d)} \vert F(h)\vert \leq 1\biggr\}.
\end{equation*}
Here, ${\rm Lip}_{1}({\mathcal C}([0,T];\RR^d))$ stands for the Lip-1 functions on ${\mathcal C}([0,T];\RR^d)$ equipped
with the metric of uniform convergence on compact subsets.

\vskip 2pt
We emphasize that $E_{1}$ is a convex closed bounded subset of $V_{1}$. Moreover, we notice that the convergence for the norm
$\| \cdot \|_{1}$ of a sequence of functions in $E_{1}$ is equivalent to the uniform convergence on compact subsets of $[0,T] \times \RR^d$. Similarly, $E_{2}$ is a convex closed bounded subset of $V_{2}$ as the convergence of non-negative measures on 
a metric space for the Kantorovitch-Rubinstein norm implies weak convergence of measures. 
We now claim:

\begin{lemma}
\label{le:phi} 
Assume that, in addition to (A1--3), the coefficients $B$ and $F$ are also bounded by $L$. Then, the mapping $\Phi: E  \ni (\varphi,\mu) \hookrightarrow(u,\PP_{X}) \in E$ defined above is continuous and has a relatively compact range.
\end{lemma}

\begin{proof}We first check the continuity of $\Phi$. Given a sequence $(\varphi^n,\mu^n)$ in $E$ converging towards $(\varphi,\mu) \in E$ with respect to the product norm on $V_{1}\times V_{2}$, and given the corresponding solutions $(X^n,Y^n,Z^n)$ and $(X,Y,Z)$ obtained by solving \eqref{fo:our_fbsde} with $(\varphi^n,\mu^n)$ and $(\varphi,\mu)$ respectively, we have (compare with 
\eqref{fo:preW2est}): $(i)$ for any $t \in [0,T]$, 
$W_{2}(\mu_{t},\mu^n_{t}) \rightarrow 0$ as $n \rightarrow + \infty$ since $(\mu^n_{t})_{n \geq 1}$ converges weakly towards $\mu_{t}$  and the moments of order $4$ of the measures $(\mu^n_{t})_{n \geq 1}$ are uniformly bounded by $\gamma'$; by boundedness of the moments of order 4 again, the integral with respect to $t$ of $W_2(\mu_{t},\mu^n_{t})$ converges towards $0$; $(ii)$ by continuity and boundedness of $\varphi$, and by a similar argument, the integral with respect to $t$ of 
$W_{2}(\varphi(t,\cdot)(\mu_{t}),\varphi(t,\cdot)(\mu_{t}^n))^2$ converges toward $0$ as $n \rightarrow + \infty$; 
$(iii)$ since the sup-norms of all the $\varphi^n$ are not greater than $\gamma$, the tightness of the measures $(\mu^n)_{n \geq 1}$ together with the uniform convergence of $(\varphi^n)_{n \geq 1}$ towards $\varphi$ on compact sets can be used to prove that 
\begin{equation*}
\lim_{n \rightarrow + \infty} \sup_{0 \leq t \leq T}
\int_{\RR^d} \vert (\varphi - \varphi^n)(t,y) \vert^2 d \mu_{t}^n(y) = 0.
\end{equation*}
Similarly, $W_{2}(\mu_{T},\mu_{T}^n) \rightarrow 0$ as $n$ tends to $+\infty$.  
From \eqref{fo:Ei} and \eqref{fo:W2est}, we obtain
\begin{equation*}
\lim_{n \rightarrow + \infty} \bigg[ {\mathbb E} \sup_{0 \leq t \leq T}\vert X_{t} - X_{t}^n \vert^2 
+ {\mathbb E}  \sup_{0 \leq t \leq T}\vert Y_{t} - Y_{t}^n \vert^2 
+ {\mathbb E} \int_{0}^T \vert Z_{t}- Z_{t}^n \vert^2 dt \biggr] = 0,
\end{equation*}
from which we deduce that $\PP_{X^n}$ converges towards $\PP_{X}$ as $n$ tend to $+ \infty$ for the topology of weak convergence of measures and thus for $\| \cdot \|_2$. Denoting by $u^n$ the FBSDE value function which is a function
from $[0,T] \times \RR^d$ into $\RR^p$ such that $Y_{t}^n = u^n(t,X_{t}^n)$, and by $u$ the FBSDE value function for which
$Y_{t} = u(t,X_{t})$, we deduce that 
\begin{equation*}
\lim_{n\rightarrow + \infty} \sup_{0 \leq t \leq T} {\mathbb E} \bigl[ \vert u^n(t,X_{t}^n) - u(t,X_{t}) \vert^2 \bigr] = 0.
\end{equation*}
By Lemma \ref{lem:1}, we know that all the mappings $(u^n)_{n \geq 1}$ are Lipschitz continuous with respect to $x$, uniformly with respect to $n$. Therefore
\begin{equation*}
\lim_{n\rightarrow + \infty} \sup_{0 \leq t \leq T} {\mathbb E} \bigl[ \vert u^n(t,X_{t}) - u(t,X_{t}) \vert^2 \bigr] = 0.
\end{equation*}
Moreover, by Arz\`ela-Ascoli's theorem, the sequence $(u^n)_{n \geq 1}$ is relatively compact for the uniform convergence on compact sets, so denoting by $\hat{u}$ the limit of a subsequence converging for the norm $\| \cdot \|_{1}$,
we deduce that, for any $t \in [0,T]$, $\hat{u}(t,\cdot)= u(t,\cdot)$ $\PP_{X_{t}}$-{\rm a.s.}. By Girsanov Theorem, 
$\PP_{X_{t}}$ is equivalent to Lebesgue measure for any $t \in (0,T]$, so that $\hat{u}(t,\cdot)=u(t,\cdot)$ for any $t \in (0,T]$. By continuity of $u$ and $\hat{u}$ on the whole $[0,T] \times \RR^d$, equality holds at $t=0$ as well. This shows that $(u^n)_{n \geq 1}$ converges towards $u$ for $\| \cdot \|_{1}$ and completes the proof of the continuity of $\Phi$. 

\vskip 2pt
We now prove that $\Phi(E)$ is relatively compact for the product norm of $V_{1}\times V_{2}$. Given $(u,\nu) = \Phi(\varphi,\mu)$ for some $(\varphi,\mu) \in E$, we know from Lemma \ref{lem:1}
that $u$ is bounded by $\gamma$ and $(1/2,1)$-H\"older continuous with respect to $(t,x)$, the H\"older constant being bounded by $\Gamma$. In particular, $u$ remains in a compact subset of ${\mathcal C}([0,T] \times \RR^d;\RR^m)$ for the topology of uniform convergence on compact sets as $(\varphi,\mu)$ varies over $E$. Similarly, $\nu$ remains in a compact set when $(\varphi,\mu)$ varies over $E$. Indeed, if $\PP_X=\nu$ is associated to $(\varphi,\mu)$, the modulus of continuity of $X$ is controlled by the fact that $B$ and $\Sigma$ are bounded by constants independent of $\varphi$ and $\mu$.
 \end{proof}

 We have completed all the steps needed to get a quick proof of the main result of this subsection.
  
 \begin{proposition}
 \label{pr:bdedexistence}
 Assume that, in additition to (A1--3), the coefficients $B$, $\Sigma$, $F$ and $G$ are bounded by $L$. Then equation \eqref{fo:fbsde} has a solution.
 \end{proposition}
 
 \begin{proof} By Schauder's fixed point theorem, $\Phi$ has a fixed point $(\varphi,\mu)$. As explained in our description of the strategy of proof, solving  \eqref{fo:our_fbsde} with this $(\varphi,\mu)$ as input, and denoting by $(X_{t},Y_{t},Z_{t})_{0 \leq  t \leq T}$ the resulting solution, by definition of a fixed point, we have $Y_{t} = \varphi(t,X_{t})$ for any $t \in [0,T]$, a.s., and $\PP_{X} = \mu$ so that $\PP_{X_{t}} = \mu_{t}$. In particular, $\varphi(t,\cdot) \diamond \mu_{t}$ coincides with 
$\PP_{(X_{t},Y_{t})}$. We conclude that $(X_{t},Y_{t},Z_{t})_{0 \leq t \leq T}$ satisfies \eqref{fo:fbsde}. 
\end{proof}

 \subsection{Relaxing the Boundedness Condition.}
We now complete the proof of Theorem \ref{th:mkv_fbsde} when the coefficients only satisfy (A1--3). The proof consists in approximating $B$ and $F$ by sequences of bounded coefficients $(B^n)_{n \geq 1}$ and $(F^n)_{n \geq 1}$. 
\begin{proof} (Theorem 1)
For any $n \geq 1$, $t \in [0,T]$, $x \in \RR^d$, $y \in \RR^p$, $z \in \RR^{p \times m}$ and $\mu \in {\mathcal P}_2(\RR^{d} \times 
\RR^m)$, we set:
\begin{equation*}
B^n(t,x,y,z,\mu) = \Pi_{n}^{(d)} \bigl( B(t,x,y,z,\mu) \bigr), 
\quad
F^n(t,x,y,z,\mu) = \Pi_{n}^{(p)} \bigl( F(t,x,y,z,\mu) \bigr), 
\end{equation*}
where, for any integer $k$, $\Pi_{n}^{(k)} $ is the orthogonal projection from $\RR^k$ onto the $k$-dimensional ball of center $0$ and of radius $n$.

For each $n$, assumptions (A1-3) are satisfied with $B^n$ and $F^n$ instead of  $B$ and $F$, and we denote by  $(X^n,Y^n,Z^n)$ the solution of 
\eqref{fo:fbsde} given by Proposition \ref{pr:bdedexistence} when the system \eqref{fo:fbsde} is driven by $B^n$, $F^n$, $\Sigma$ and $G$. As explained in the previous subsection, the process $Y^n$ satisfies 
$Y^n_{t} = u^n(t,X_{t}^n)$, for any $t \in [0,T]$, for some deterministic function $u^n$. The first step of the proof is to establish the relative compactness of the families $(u^n)_{n \geq 1}$ and $(\PP_{X^n})_{n \geq 1}$.

We notice first that the processes $(Y^n)_{n \geq 1}$ are uniformly bounded by a constant that depends upon $L$ only. Indeed, applying It\^o's formula and using the specific growth condition (A2), we get:
\begin{equation*}
\forall t \in [0,T], \quad {\mathbb E} \bigl[ \vert Y_{t}^n \vert^2 \bigr] \leq C + C \int_{t}^T {\mathbb E} \bigl[ \vert Y_{s}^n \vert^2 \bigr] ds,
\end{equation*}
for some constant $C$ depending on $T$ and $L$ only. As usual, the value of $C$ may vary from line to line. By Gronwall's lemma, we deduce that 
the quantity $\sup_{0 \leq t \leq T}
{\mathbb E} [ \vert Y_{t}^n \vert^2]$ can be bounded in terms of $T$ and $L$ only. Injecting this estimate into (A2) shows that the 
driver $(-F^n(t,X_{t}^n,Y_{t}^n,Z_{t}^n,\PP_{(X_{t}^n,Y_{t}^n)}))_{0 \leq t \leq T}$  is bounded by $(C(1 + \vert Y_{t}^n \vert))_{
0 \leq t \leq T}$, for a possibly new value of $C$. Using now the uniform boundedness of the function $G$ giving the terminal condition $Y_T$, we conclude that the processes $(Y^n_{t})_{0 \leq t \leq T}$ are bounded, uniformly in $n \geq 1$. Indeed, for any $n \geq 1$, the process
$(\vert Y^n_{t} \vert^2)_{0 \leq t \leq T}$ can be seen as a solution of a BSDE with random coefficients; thus, it can be compared with the deterministic solution of a deterministic BSDE with a constant terminal condition.  In particular, there exists a positive constant, say $\gamma$, such that, for any $n \geq 1$ for any $t \in [0,T]$, $\vert u^n(t,X_{t}^n) \vert \leq \gamma$. Since for each $t\in[0,T]$, $X^n_t$ has a density with respect to Lebesgue's measure on $\RR^d$ (recall that $B^n$ is bounded and $\Sigma\Sigma^\dagger$ is bounded from above and from below away from $0$), we deduce that $u^n$ is bounded by $\gamma$ on the whole $[0,T] \times \RR^d$, and for any $n \geq 1$.  As a by-product, we get
\begin{equation}
\label{fo:zmoment}
{\mathbb E} \biggl[ \biggl( \int_{0}^T \vert Z_{s}^n \vert^2 ds \biggr)^2 \biggr] \leq C.
\end{equation}

When $p=1$, Theorem 2.1 in 
 \cite{DelarueGuatteri} says that the value function $u^n$ is a continuous solution with Sobolev derivatives of order $1$ in time and of order $2$ in space of a quasilinear PDE. As mentioned in the conclusion of \cite{DelarueGuatteri}, the result remains true when $p \geq 2$, the value function $u^n$ then solving a system of quasilinear PDEs. Using the fact that the proofs of  Theorems 1.1 and 1.3 of \cite{Delarue03} are local (even though they are stated under global assumptions), 
we deduce that the functions $(u^n)_{n \geq 1}$ are continuous in $(t,x)$, uniformly in $n \geq 1$,  on any compact subsets of $[0,T] \times \RR^d$. In a similar way, by Theorem 2.7 of \cite{Delarue03}, the proof of which is also local, the functions $((u^n(t,\cdot))_{n\geq 1})_{0 \leq t \leq T}$ are locally Lipschitz-continuous, uniformly in $t \in [0,T]$ and in $n \geq 1$.

Applying It\^o's formula to $\vert X^n \vert^2$, using the growth conditions (A2), the uniform boundedness of $Y^n_t$, the boundedness of $\Sigma$ and the bound \eqref{fo:zmoment}, we can use Gronwall's lemma and get the existence of a finite constant $C$ such that, for any $n \geq 1$, 
\begin{equation}
\label{fo:xmoment}
{\mathbb E}[\sup_{0 \leq t \leq T} \vert X_{t}^n \vert^4] \leq C. 
\end{equation}
Using the bound (A2) for $B^n$ and the same constant $L$, together with the uniform boundedness of the paths $Y^n_t$, \eqref{fo:zmoment} and \eqref{fo:xmoment}, it is easy to check that
$$
\EE[|X^n_t-X^n_s|^4]\le C|t-s|^2
$$
for all $s,t\in[0,T]$ for a constant $C$ independent of $n$, $s$ and $t$. Consequently, Kolmogorov's criterion shows that the family $(\PP_{X^n})_{n \geq 1}$ of probability measures on ${\mathcal C}([0,T];\RR^d)$ is tight. Replacing $Y^n_t$ by $u^n(t,X^n_t)$ in $\Sigma(t,X^n_t,Y^n_t,\PP_{(X^n_{t},Y^n_{t})})$, we see that each component of $u^n$ satisfies a PDE, and since $\Sigma(t,x,u^n(t,x),\PP_{(X^n_{t},Y^n_{t})})$ is locally H\"older continuous uniformly in $n\ge 1$, we can use Theorem 4, Chapter 7 Section 2 of \cite{Friedman} and conclude that the gradients of these components are locally H\"older continuous in $(t,x)$ on compact subsets of $[0,T) \times \RR^d$, uniformly in $n\ge 1$. Consequently, one can thus extract a subsequence $(n_{k})_{k\ge 1}$ such that 
$u^{n_{k}}$ and $\partial_xu^{n_{k}}$ converge uniformly on compact subsets of
$[0,T] \times \RR^d$ and of  
 $[0,T) \times \RR^d$ respectively, 
and $\PP_{X^{n_{k}}}$ converges towards a probability measure $\mu$ on ${\mathcal C}([0,T];\RR^d)$.
If we denote by $u$ the limit of  $u^{n_{k}}$, the function $u$ is continuously differentiable with respect to $x$ on $[0,T) \times \RR^d$, and
$\partial_{x} u^{n_{k}}$ converges towards $\partial_{x} u$ uniformly on compact subsets of $[0,T) \times \RR^d$. 
As before, we can use the uniform boundedness of the functions $u^n$, the tightness of the sequence $(\PP_{X^{n_{k}}})_{k\ge 1}$ and the uniform convergence of $u^{n_{k}}$ towards $u$ on any compact subset of $[0,T) \times \RR^d$ to conclude that for any $t \in [0,T]$,
\begin{equation}
\label{fo:was}
\lim_{p \rightarrow + \infty} W_{2} \bigl( u^{n_{k}}(t,\cdot) \diamond \PP_{X^{n_{k}}_{t}}, u(t,\cdot) \diamond \mu_{t} \bigr)= 0, 
\end{equation}
where $\mu_{t}$ stands for the marginal law of $\mu$ of time index $t$. 
Now for each $R > 0$, we denote by $\tau_{R}^n$ the first time the process $\u X^n$ leaves the ball of center $0$ and radius $R$, in other words $\tau_{R}^n = \inf\{t \geq 0;\; \vert X_{t}^n \vert \geq R\}$. 
Using the uniform convergence of $(u^{n_{k}},\partial_{x} u^{n_{k}})$ to $(u,\partial_{x}u)$ on compact subsets of $[0,T) \times \RR^d$ together with \eqref{fo:zmoment}, we have for each fixed $R>0$:
\begin{equation*}
\lim_{k \rightarrow + \infty} \sup_{\ell,\ell' \geq 0} {\mathbb E} \bigl[ \sup_{0 \leq t \leq \tau^{n_{k+\ell}}_{R} \wedge \tau^{n_{k+\ell'}}_{R}}
\vert X_{t}^{n_{k+\ell}} - X_{t}^{n_{k+\ell'}} \vert^2 \bigr] = 0.
\end{equation*}
Since \eqref{fo:xmoment} implies that:
\begin{equation*}
\lim_{R \rightarrow + \infty} \sup_{\ell,\ell' \geq 1} {\mathbb E}\bigl[ \sup_{0 \leq t \leq T} \vert X_{t}^{n_{k+\ell'}} - X_{t}^{n_{k+\ell}} \vert^2 
{\mathbf 1}_{\{T > \tau^{n_{k+\ell}}_{R} \wedge \tau^{n_{k+\ell'}}_{R}\}} \bigr] = 0,
\end{equation*}
we must have: 
\begin{equation*}
\lim_{k \rightarrow + \infty} \sup_{\ell,\ell' \geq 0} {\mathbb E} \bigl[ \sup_{0 \leq t \leq T}
\vert X_{t}^{n_{k+\ell}} - X_{t}^{n_{k+\ell'}} \vert^2 \bigr] = 0,
\end{equation*}
which shows that the sequence $(\u X^{n_{k}})_{k \geq 1}$ is a Cauchy sequence. We denote by $\u X$ the limit. 
Since for each $n\ge 1$ we have:
\begin{equation*}
dX^n_{t} = B^n \bigl(t,X^n_{t},u^n(t,X^n_{t}),v^n(t,X^n_{t}),u^n(t,\cdot) \diamond \PP_{X^n_{t}} \bigr) dt + \Sigma(t,X^n_{t},u^n(t,X^n_{t}),u^n(t,\cdot) \diamond \PP_{X^n_{t}}) dW_{t},
\end{equation*}
for $0 \leq t \leq T$, 
with $v^n(t,x)=\partial_{x}u^n(t,x) \Sigma(t,x,u^n(t,x),u^n(t,\cdot)\diamond \PP_{X_{t}^n})$,
we can take the limit along the subsequence $(n_{k})_{k\ge 1}$ and since $\mu$ is the law of $\u X$, the local uniform convergence of $u^{n_k}$ and $\partial_xu^{n_k}$ towards $u$ and $\partial_xu$, together with \eqref{fo:was} imply that
\begin{equation*}
dX_{t} = B \bigl(t,X_{t},u(t,X_{t}),v(t,X_{t}),u(t,\cdot) \diamond \PP_{X_{t}} \bigr) dt + \Sigma(t,X_{t},u(t,X_{t}),u(t,\cdot) \diamond \PP_{X_{t}}) dW_{t}
\end{equation*}
for $0 \leq t \leq T$, with $v(t,x)=\partial_{x}u(t,x) \Sigma(t,x,u(t,x),u(t,\cdot)\diamond \PP_{X_{t}})$,
which is exactly the forward component of the McKean-Vlasov FBSDE \eqref{fo:fbsde} provided we set $Y_{t}=u(t,X_{t})$ and $Z_{t} = v(t,X_{t})$ for $t \in [0,T)$.
It is plain to deduce that the sequences $(\u Y^{n_{k}})_{k \geq 1}$ and $(\u Z^{n_{k}})_{k \geq 1}$ are Cauchy sequences for the norms ${\mathbb E}[\sup_{0 \leq s\leq T} \vert \cdot_{s} \vert^2 ]^{1/2}$ and $\EE[\int_{0}^T \vert \cdot_{s} \vert^2 ds ]^{1/2}$ respectively. Denoting the respective limits by $\u Y$ and $\u Z$, it holds, $\PP$-a.s., for any $t \in [0,T]$, $Y_{t}=u(t,X_{t})$, and, $\PP \otimes dt$ a.e., $Z_{t} = v(t,X_{t})$. Passing to the limit in \eqref{fo:fbsde} with $(B,F)$ therein replaced by $(B^{n_{k}},F^{n_{k}})$, we deduce that $(\u X,\u Y,\u Z)$ satisfies \eqref{fo:fbsde}. 
\end{proof} 

\subsection{Counter-Example to Uniqueness} 
We close this section with a counter-example showing that uniqueness cannot hold in general under assumptions (A1--3), even in the case $d=m=p=1$. Indeed, let us consider the forward-backward system
\begin{equation}
\label{eq:25:3:5}
\begin{split}
&dX_{t} = B(\EE(Y_{t})) dt + dW_{t},\qquad X_0=x_0,
\\
&dY_{t} = - F(\EE(X_{t})) dt + Z_{t} dW_{t},\qquad Y_{T} = G(\EE(X_{T})),
\end{split}
\end{equation}
where $B$, $F$ and $G$ are real valued bounded and Lipschitz-continuous functions on the real line. Let us also assume that they coincide with the identity on $[-R,R]$ for $R$ large enough. In other words, we assume that $B(x)=F(x)=G(x)=x$ for $|x|\le R$.  For $T=\pi/4$ and for any $A \in \RR$, the pair
\begin{equation*}
\begin{split}
&x_{t} = A  \sin(t),
\quad
y_{t} = A \cos(t), \quad 0 \leq t \leq T = \frac{\pi}{4},
\end{split}
\end{equation*}
satisfies
\begin{equation*}
\begin{split}
&\dot{x}_{t} = y_{t},
\quad \dot{y}_{t} = -x_{t}, \quad 0 \leq t \leq T,
\\
&y_{T} = x_{T},
\end{split}
\end{equation*}
with $x_{0}=0$ as initial condition. Therefore, for $\vert A \vert \leq R$, $(A \sin(t),A \cos(t))_{0 \leq t \leq T}$ is a solution to the deterministic forward-backward system
\begin{equation*}
\begin{split}
&\dot{x}_{t} = B(y_{t}),
\quad \dot{y}_{t} = - F(x_{t}), \quad 0 \leq t \leq T,
\\
&y_{T} = G(x_{T}),
\end{split}
\end{equation*}
with $x_{0}=0$ as initial condition. For such a value of $A$, set now:
\begin{equation*}
X_{t} = x_{t} + W_{t},
\quad
Y_{t} = y_{t}, \quad 0 \leq t \leq T.
\end{equation*}
Then, $(X,Y,0)$ solves 
\begin{equation*}
\begin{split}
&dX_{t} = B(\EE(Y_{t})) dt + dW_{t}
\\
&dY_{t} = - F(\EE(X_{t})) dt + 0 \ dW_{t},
\end{split}
\end{equation*}
with $X_{0}=0$ and $Y_{T} = G(\EE(X_{T}))$, so that uniqueness fails.

\begin{remark}
The reason for the failure of uniqueness can be explained as follows. In the standard framework, as explained in \cite{Delarue02}, uniqueness holds because of the smoothing effect of the diffusion operator in the spatial direction. However, in the McKean-Vlasov setting, the smoothing effect of the diffusion operator is ineffective in the direction of the measure variable. 
\end{remark}

\section{\textbf{Applications}}
\label{se:applications}

We now describe the set-up of the two applications mentioned in the introduction.
\vskip 2pt 
We define the set $\bA$ of admissible controls as the set of progressively measurable processes $\u \alpha=(\alpha_t)_{0\le t\le T}$ with values in a measurable space $(A,\cA)$. Typically, $A$ is a Borel subset of a Euclidean space $\RR^k$, and $\cA$ the $\sigma$-field induced by the Borel $\sigma$-field of this Euclidean space.
For each admissible control process $\u\alpha$ we consider the solution 
$\u X=(X_t)_{0\le t\le T}$ of the (nonlinear) stochastic differential equation of McKean-Vlasov type
\begin{equation}
\label{fo:mkvsde}
dX_t=b(t,X_t,\PXt, \alpha_t)dt+\sigma(t,X_t,\PXt,\alpha_t)dW_t\qquad 0\le t\le T,\quad X_0=x_0,
\end{equation}
where the drift $b$ and the volatility $\sigma$ are deterministic measurable functions $b:[0,T]\times\RR^d\times \cP_2(\RR^d)\times A\rightarrow \RR^{d}$ and  
$\sigma:[0,T]\times\RR^d\times \cP_2(\RR^d)\times A\rightarrow \RR^{d\times m}$ satisfying regularity conditions to be specified below. The stochastic optimization problem associated to the optimal control of McKean-Vlasov stochastic dynamics is to minimize the objective function
\begin{equation}
\label{fo:mkvobjective}
J(\alpha)=\EE\left\{\int_0^Tf(t,X_t,\PP_{X_t}, \alpha_t)dt+g(X_T,\PP_{X_T})\right\}  ,
\end{equation}
over a set $\u\alpha\in\bA$  where the running cost function $f$ is a real valued deterministic measurable function defined on $[0,T]\times\RR^d\times \cP_2(\RR^d)\times A$ and  the terminal cost function $g$  is an $\RR^d$--valued deterministic measurable function defined on $\RR^d\times \cP_2(\RR^d)$. Though different in its goal, the mean field problem as stated by Lasry and Lions is similar in many ways. The main difference comes from the fact that, because the ultimate goal is to identify approximate Nash equilibriums for large games, the stochastic optimization problem is to minimize the objective function
\begin{equation}
\label{fo:mfgobjective}
J(\alpha)=\EE\left\{\int_0^Tf(t,X_t,\mu_t, \alpha_t)dt+g(X_T,\mu_T)\right\}  ,
\end{equation}
over the same set $\bA$ of admissible control processes when the measure argument in the dynamics of the state, the running cost function $f$ and  the terminal cost function $g$  are fixed -- one often say \emph{frozen} to emphasize this difference -- and given by a deterministic flow $\u\mu=(\mu_t)_{0\le t\le T}$ of probability measures on $\RR^d$. The similarity comes from the fact that, once the optimization for $\u\mu$ fixed is accomplished, this flow of measures is determined in such a way that the marginal distribution of the solution of the stochastic differential equation
\begin{equation}
\label{fo:mfgsde}
dX_t=b(t,X_t,\mu_t, \alpha_t)dt+\sigma(t,X_t,\mu_t,\alpha_t)dW_t\qquad 0\le t\le T,\quad X_0=x_0,
\end{equation}
is $\mu_t$, in other words in such a way that $\mu_t=\PP_{X_t}$ for all $0\le t\le T$, forcing the optimally controlled state process $\u X$ to actually solve the McKean - Vlasov equation \eqref{fo:mkvsde}. If an appropriate stochastic version of the Pontryagin maximum can be proven, the probabilistic approach to these stochastic optimization problems is to introduce adjoint processes $(\u Y,\u Z)=(Y_t,Z_t)_{0\le t\le T}$ solving the \emph{adjoint} BSDE, and to couple these backward equations with the forward equations \eqref{fo:mfgsde} and \eqref{fo:mkvsde} by introducing a  control $\u\alpha$ minimizing a specific Hamiltonian function. This coupling of the forward and backward equations creates a FBSDE of the McKean-Vlasov type in which the marginal distributions of the solutions appear in the coefficients, and which needs to be solved in order to find a solution of the stochastic optimization problems. This was the motivation behind the statement of the problem considered in Section \ref{se:fbsde}.

\vskip 4pt
For the sake of completeness we assume that, in both applications, the drift coefficient $b$ and the volatility matrix $\sigma$ satisfy the following assumptions.
\begin{itemize}
\item[(S1)] \hskip 6pt The functions $(b(t,0,\delta_{0},0))_{0\le t\le T}$ and $(\sigma(t,0,\delta_{0},0))_{0\le t\le T}$ are continuous;
\item[(S2)] \hskip 6pt  $\exists c>0$, $\forall t\in[0,T]$,  $\forall \alpha\in A$,  $\forall x,x'\in\RR^d$,  $\forall \mu,\mu'\in\cP_2(\RR^d)$,
$$
|b(t,x,\mu,\alpha)-b(t,x',\mu',\alpha)| +  |\sigma(t,x,\mu,\alpha)-\sigma(t,x',\mu',\alpha)|\le c(|x-x'|+W_2(\mu,\mu'))
$$
\end{itemize} 
We further restrict the set $\bA$ to the set of progressively measurable processes $\u\alpha$ 
in ${\mathbb H}^{2,k}$. 
Together with the Lipschitz assumption (S2), this definition guarantees that, for any $\u\alpha\in\bA$, there exists a unique solution $\u X=\u X^{\u\alpha}$ of both \eqref{fo:mfgsde} and \eqref{fo:mkvsde} and that moreover, this solution satisfies
$$
\EE \bigl[ \sup_{0\le t\le T}|X_t|^q \bigr] <\infty
$$
for every $q \geq 1$. See for example \cite{Sznitman,JourdainMeleardWoyczynski} for a proof for the McKean-Vlasov equations.

\subsection{Solvability of FBSDE for Mean-Field Games}
\label{sub:mfg}
In this subsection, we apply the abstract existence result of Section \ref{se:fbsde} to the Mean Field Game (MFG) problem described in the introduction.
As in \cite{CarmonaDelarue2}, we assume that the volatility function is a constant matrix $\sigma$ (of size $d \times m$), but here we assume in addition that ${\rm det}(\sigma \sigma^{\dagger}) >0$. 
Since the stochastic optimization problem is solved after the flow of measures is frozen, for each fixed $\mu$, the Hamiltonian of the system reads:
\begin{equation}
\label{fo:mfg_hamiltonian}
H^\mu(t,x,y,\alpha)=b(t,x,\mu,\alpha)\cdot y + f(t,x,\mu,\alpha), \quad x,y \in \RR^d, \ \alpha \in A, \ \mu \in {\mathcal P}_{2}(\RR^k),
\end{equation}
`$\cdot$' standing for the inner product.
For each $\mu$, we assume the existence of a function $(t,x,y)\hookrightarrow\hat{\alpha}^\mu(t,x,y) \in A$ 
which  is Lipschitz-continuous with respect to $(x,y,\mu)$, uniformly in $t \in [0,T]$, the Lipschitz property holding true with respect to to the product distance of the Euclidean distances on $\RR^{d} \times \RR^d$ and the $2$-Wasserstein distance on ${\mathcal P}_{2}(\RR^d)$ and such that:
\begin{equation}
\label{fo:alphahat}
\hat{\alpha}^\mu(t,x,y) \in \textrm{argmin}_{\alpha \in \RR^d} H^\mu(t,x,y,\alpha), \quad t \in [0,T], \ x,y \in \RR^d, 
\ \mu \in {\mathcal P}_{2}(\RR^d).
\end{equation}
The existence of such a function was proven in \cite{CarmonaDelarue2} under specific assumptions on the drift $b$ and the running cost function $f$ used there.
Using a standard version of the stochastic maximum principle, and coupling the forward dynamics and the adjoint BSDE by plugging the optimizer \reff{fo:alphahat} lead to the solution, for each frozen flow $\u\mu=(\mu_t)_{0\le t\le T}$ of measures, of a standard FBSDE. Now, if we add the requirement that $\mu_t$ should coincide for each $t$ with the marginal distribution of the optimally controlled state, the solution of the MFG stochastic optimization problem reduces to the solution of the FBSDE of McKean - Vlasov type
\begin{equation}
\label{fo:mfg_fbsde}
\begin{split}
&dX_{t} = b \bigl( t,X_{t},\PP_{X_{t}},\hat{\alpha}(t,X_{t},Y_{t},\PP_{X_{t}})\bigr) dt + \sigma dW_{t},
\\
&dY_{t} = - \partial_{x} f\bigl(t,X_{t},\PP_{X_{t}},\hat{\alpha}(t,X_{t},Y_{t},\PP_{X_{t}}) \bigr) dt 
 - \partial_{x} b\bigl(t,X_{t},\PP_{X_{t}},\hat{\alpha}(t,X_{t},Y_{t},\PP_{X_{t}}) \bigr) \cdot Y_t dt + Z_{t}dW_{t},
\end{split}
\end{equation}
with initial condition $X_{0}=x_{0}$ for a given deterministic point $x_{0} \in \RR^d$, and terminal condition $Y_{T} = \partial_{x} g(X_{T},\PP_{X_{T}})$.  
Existence of a solution for this McKean-Vlasov FBSDE follows by applying Theorem \ref{th:mkv_fbsde} if we assume that $\partial_{x} f$ and $\partial_{x} g$ are bounded and Lipschitz-continuous and set: 
\begin{equation*}
\begin{split}
&B\bigl(t,X_{t},Y_{t},Z_{t},\PP_{(X_{t},Y_{t})}\bigr) = b(t,X_{t},\PP_{X_{t}},\alpha_t),
\\
&F\bigl(t,X_{t},Y_{t},Z_{t},\PP_{(X_{t},Y_{t})}\bigr) = \partial_{x} f\bigl(t,X_{t},\PP_{X_{t}},\hat{\alpha}(t,X_{t},Y_{t},\PP_{X_{t}}) \bigr)\\
&\phantom{??????????????????????????????}+\partial_{x} b\bigl(t,X_{t},\PP_{X_{t}},\hat{\alpha}(t,X_{t},Y_{t},\PP_{X_{t}}) \bigr) \cdot Y_{t},
\\
&G(X_{T},\PP_{X_{T}}) = \partial_{x} g(X_{T},\PP_{X_{T}}).
\end{split}
\end{equation*}
While the result of the present note provides existence in quite a general set up, \cite{CarmonaDelarue2} also 
allows for running costs with at most linear growth in $x$
(and provides much more in terms of the identification of approximate Nash equilibriums). However, the drift $b$ needs to be of a very specific \emph{affine form}, namely $b(t,x,y,\alpha,\mu)=b_0(t,\mu)+b_1(t)x+b_2(t)\alpha$ for some deterministic functions $b_0$, $b_1$ and $b_2$, and the running  cost function $f$ has to satisfy a strong convexity assumption, the latter
 allowing $\sigma$ to be degenerate, providing a more efficient approximation procedure to reduce the result to the bounded case and ensuring the validity of the converse of the stochastic maximum principle.

\subsection{Optimal Control of McKean-Vlasov Stochastic Dynamics}
\label{sub:mkv_control}
Finally, we explain how the existence result of this paper generalizes the existence result of \cite{CarmonaDelarue3} where a solution of the optimal control of stochastic differential equations of the  McKean-Vlasov type is given.

As before, we assume that the volatility function is a constant matrix $\sigma$ such that 
${\rm det}(\sigma \sigma^{\dagger}) >0$, so the Hamiltonian has the form 
$$
H(t,x,y,\mu,\alpha) = b(t,x,\mu,\alpha)\cdot y  + f(t,x,\mu,\alpha),
$$
for $t \in [0,T]$, $x,y \in \RR^d$, $\alpha \in A$ and $\mu \in {\mathcal P}_{2}(\RR^d)$. 
Again, we assume the existence of a function $(t,x,y,\mu)\hookrightarrow\hat{\alpha}(t,x,y,\mu) \in A$ 
which  is Lipschitz-continuous with respect to $(x,y,\mu)$, uniformly in $t \in [0,T]$, the Lipschitz property holding true with respect to the product distance of the Euclidean distances on $\RR^{d} \times \RR^d$ and the $2$-Wasserstein distance on ${\mathcal P}_{2}(\RR^d)$ and such that:
\begin{equation}
\label{fo:alphahat'}
\hat{\alpha}(t,x,y,\mu) \in \textrm{argmin}_{\alpha \in \RR^d} H(t,x,y,\mu,\alpha), \quad t \in [0,T], \ x,y \in \RR^d, 
\ \mu \in {\mathcal P}_{2}(\RR^d).
\end{equation}
The existence of such a function was proven in \cite{CarmonaDelarue3} under specific assumptions on the drift $b$ and the running cost function $f$ used there. However, the major difference with the mean field game problem comes from the form of the adjoint equation which now involves differentiation of the Hamiltonian with respect to the measure parameter. A special form of adjoint equation was introduced, and a new stochastic maximum principle was proven in \cite{CarmonaDelarue3}. Once this new form of adjoint equation is coupled with the forward dynamical equation through the plugged-in optimal control feedback $\hat \alpha$ defined in \reff{fo:alphahat'}, the associated McKean-Vlasov FBSDE takes the form
\begin{equation}
\label{eq:24:12:1:b}
\begin{split}
&dX_{t} = b \bigl( t,X_{t},\PP_{X_{t}},\hat{\alpha}(t,X_{t},Y_{t},\PP_{X_{t}})\bigr) dt + \sigma dW_{t},
\\
&dY_{t} = - \partial_{x} f\bigl(t,X_{t},\PP_{X_t},\hat{\alpha}(t,X_{t},Y_{t},\PP_{X_{t}}) \bigr) dt \\
&\hspace{30pt} - \partial_{x} b\bigl(t,X_{t},\PP_{X_t},\hat{\alpha}(t,X_{t},Y_{t},\PP_{X_{t}}) \bigr) \cdot Y_t  dt 
+ Z_{t}dW_{t}\\
&\hspace{30pt} - \tilde{\mathbb E} \bigl[\partial_{\mu} f\bigl(t,\tilde{X}_{t},\PP_{X_{t}},\hat{\alpha}(t,\tilde{X}_{t},\tilde{Y}_{t},\PP_{X_{t}}) \bigr)(X_{t}) \bigr] dt\\
&\hspace{30pt} - \tilde{\mathbb E} \bigl[\partial_{\mu} b\bigl(t,\tilde{X}_{t},\PP_{X_{t}},\hat{\alpha}(t,\tilde{X}_{t},\tilde{Y}_{t},\PP_{X_{t}}) \bigr)(X_{t})\cdot Y_t\bigr] dt, 
\end{split}
\end{equation}
with initial condition $X_{0}=x_{0}$ for a given deterministic point $x_{0} \in \RR^d$, and terminal condition $Y_{T} = \partial_{x} g(X_{T},\PP_{X_{T}}) + \tilde{\EE}[\partial_{\mu} g(\tilde{X}_{T},\PP_{X_{T}})(X_{T})]$.  Above, the tilde refers to an independent copy of $(X_t, Y_t, Z_t)$.
The notations $\partial_{\mu} b(t,x,\PP_{\xi},\alpha)(\xi)$ and 
$\partial_{\mu} f(t,x,\PP_{\xi},\alpha)(\xi)$, for an $\RR^d$--valued random variable $\xi$ of order $2$, 
represent  the derivatives of $b$ and $f$ with respect to the measure argument at $(t,x,\PP_{\xi},\alpha)$
according to the rule of differentiation introduced by P.L. Lions as explained in \cite{Cardaliaguet}: the derivatives are represented by random variables, obtained by plugging the current value of $\xi$ into some functions $\partial_{\mu} b(t,x,\PP_{\xi},\alpha)(\,\cdot\,)$ and $\partial_{\mu} f(t,x,\PP_{\xi},\alpha)(\,\cdot\,)$.  We chose the order in which the various terms appear in the right hand side of the backward equation in order to emphasize the last two terms which are not present in standard forms of adjoint BSDEs. Because of these two last terms, the coefficients depend on the joint law of $(X_{t},Y_{t})$, which is not the case in \eqref{fo:mfg_fbsde}. Despite this non-standard form of the FBSDE, the existence result of this note applies (under appropriate conditions on the coefficients). Under very restrictive conditions on the drift $b$ and a strict convexity assumption on $f$, \cite{CarmonaDelarue3} provides uniqueness, allows for at most linear growth in the variable $x$ and guarantees the validity of the converse of the maximum principle.

\bibliographystyle{plain}

\begin{thebibliography}{10}

\bibitem{Bensoussanetal}
A.~Bensoussan, K.C.J. Sung, S.C.P. Yam, and S.P. Yung.
\newblock Linear quadratic mean field games.
\newblock Technical report, 2011.

\bibitem{Cardaliaguet}
P.~Cardaliaguet.
\newblock Notes on mean field games.
\newblock Technical report, 2010.

\bibitem{CarmonaDelarue3}
R.~Carmona and F.~Delarue.
\newblock Optimal control of {M}c{K}ean-{V}lasov stochastic dynamics.
\newblock Technical report, 2012.

\bibitem{CarmonaDelarue2}
R.~Carmona and F.~Delarue.
\newblock Probabilistic analysis of mean field games.
\newblock Technical report, 2012.

\bibitem{CarmonaDelarueLachapelle}
R.~Carmona, F.~Delarue, and A.~Lachapelle.
\newblock Control of {M}c{K}ean-{V}lasov versus {M}ean {F}ield {G}ames.
\newblock {\em Mathematical Financial Economics}, 2012.

\bibitem{Delarue02}
F.~Delarue.
\newblock On the existence and uniqueness of solutions to {FBSDE}s in a
  non-degenerate case.
\newblock {\em Stochastic Processes and Applications}, 99:209--286, 2002.

\bibitem{Delarue03}
F.~Delarue.
\newblock Estimates of the solutions of a system of quasi-linear {PDE}s. a
  probabilistic scheme.
\newblock S{\'{e}}minaire de Probabilit{\'{e}}s XXXVII, pages 290--332.
  Springer Verlag, 2003.

\bibitem{DelarueGuatteri}
F.~Delarue and G.~Guatteri.
\newblock Weak existence and uniqueness for fbsdes.
\newblock {\em Stochastic Processes and Applications}, 116:1712--1742, 2006.

\bibitem{Friedman}
Avner Friedman.
\newblock {\em Partial differential equations of parabolic type}.
\newblock Prentice-Hall, Englewood Cliffs, N.J., first edition, 1964.

\bibitem{JourdainMeleardWoyczynski}
B.~Jourdain, S.~Meleard, and W.~Woyczynski.
\newblock Nonlinear {S}{D}{E}s driven by {L}\'evy processes and related
  {P}{D}{E}s.
\newblock {\em ALEA, Latin American Journal of Probability}, 4:1--29, 2008.

\bibitem{Kac}
M.~Kac.
\newblock Foundations of kinetic theory.
\newblock In {\em Proceedings of the 3rd Berkeley Symposium on Mathematical
  Statistics and Probability}, volume~3, pages 171 -- 197, 1956.

\bibitem{Kac.book}
M.~Kac.
\newblock {\em Probability and Related Topics in the Physical Sciences}.
\newblock Interscience, 1958.

\bibitem{MFG1}
J.M. Lasry and P.L.Lions.
\newblock Jeux {\`{a}} champ moyen i. le cas stationnaire.
\newblock {\em Comptes Rendus de l'Acad{\'{e}}mie des Sciences de Paris, ser.
  A}, 343(9), 2006.

\bibitem{MFG2}
J.M. Lasry and P.L.Lions.
\newblock Jeux {\`{a}} champ moyen ii. horizon fini et contr{\^{o}}le optimal.
\newblock {\em Comptes Rendus de l'Acad{\'{e}}mie des Sciences de Paris, ser.
  A}, 343(10), 2006.

\bibitem{MFG3}
J.M. Lasry and P.L.Lions.
\newblock Mean field games.
\newblock {\em Japanese Journal of Mathematics}, 2(1), Mar. 2007.

\bibitem{Sznitman}
A.S. Sznitman.
\newblock Topics in propagation of chaos.
\newblock In {\em D. L. Burkholder et al. , Ecole de Probabilit\'es de Saint
  Flour, XIX-1989}, volume 1464 of {\em Lecture Notes in Mathematics}, pages
  165--251, 1989.

\bibitem{Yong3}
J.~Yong.
\newblock A linear-quadratic optimal control problem for mean field stochastic
  differential equations.
\newblock Technical report, 2011.

\end{thebibliography}

\end{document}